\documentclass[12pt]{amsart}
\usepackage{amscd,amsmath,amssymb,amsfonts}
\usepackage{dsfont,enumerate}
\usepackage[cmtip, all]{xy}

\usepackage{graphicx,amssymb,amsmath,amsfonts,amsthm,amscd,hyperref,textcomp,relsize,colortbl}
\usepackage{tikz-cd}
\usetikzlibrary{babel}

\theoremstyle{plain}
\newtheorem{thm}{Theorem}
\newtheorem{lem}[thm]{Lemma}
\newtheorem{cor}[thm]{Corollary}
\newtheorem{prop}[thm]{Proposition}

\theoremstyle{definition}

\newtheorem{rmk}[thm]{Remark}

\numberwithin{thm}{section}
\numberwithin{equation}{thm}

\newcommand{\rank}{{\rm rank}}

\newcommand{\Gal}{{\rm Gal}}

\newcommand{\Trace}{{\rm Trace}}

\newcommand{\sA}{{\mathcal A}}
\newcommand{\sB}{{\mathcal B}}

\newcommand{\sF}{{\mathcal F}}
\newcommand{\sG}{{\mathcal G}}
\newcommand{\sH}{{\mathcal H}}

\newcommand{\sK}{{\mathcal K}}
\newcommand{\sL}{{\mathcal L}}

\newcommand{\A}{{\mathbb A}}

\newcommand{\C}{{\mathbb C}}

\newcommand{\F}{{\mathbb F}}
\newcommand{\G}{{\mathbb G}}

\renewcommand{\P}{{\mathbb P}}
\newcommand{\Q}{{\mathbb Q}}

\newcommand{\Z}{{\mathbb Z}}
\newcommand{\bfZ}{{\mathbf Z}}

\newcommand{\triv}{{\mathds{1}}}
\newcommand{\ord}{\mathrm {ord}}

\newcommand{\CC}{\mathbb C}

\newcommand{\Fp}{{\mathbb F}_p}

\newcommand{\Aut}{\mathrm{Aut}}

\newcommand{\SL}{\mathrm{SL}}
\newcommand{\SO}{\mathrm{SO}}
\newcommand{\PGL}{\mathrm{PGL}}
\newcommand{\PSL}{\mathrm{PSL}}
\newcommand{\Co}{\mathrm{Co}}
\newcommand{\Alt}{\mathsf {A}}
\newcommand{\SSS}{\mathsf {S}}

\newcommand{\Teich}{{\mathsf {Teich}}}
\newcommand{\Swan}{{\mathsf {Swan}}}
\newcommand{\Gauss}{{\mathsf {Gauss}}}
\newcommand{\lbr}{\boldsymbol{[}}
\newcommand{\rbr}{\boldsymbol{]}}

\begin{document}
\title[Rigid local systems with monodromy group $\Co_2$]{Rigid local systems with monodromy group the Conway group $\Co_2$}
\author{Nicholas M. Katz, Antonio Rojas-Le\'{o}n, and Pham Huu Tiep}
\address{Department of Mathematics, Princeton University, Princeton, NJ 08544, USA}
\email{nmk@math.princeton.edu}
\address{Departamento de \'{A}lgebra, Universidad de Sevilla, c/Tarfia s/n, 41012 Sevilla, Spain}
\email{arojas@us.es}
\address{Department of Mathematics, Rutgers University, Piscataway, NJ 08854, USA}
\email{tiep@math.rutgers.edu}
\thanks{The second author was partially supported by MTM2016-75027-P (Ministerio de Econom\'ia y Competitividad) and FEDER. The third author gratefully acknowledges the support of the NSF (grant 
DMS-1840702).} 

\maketitle

\begin{abstract} 
We first develop some basic facts about hypergeometric sheaves on the multiplicative group $\G_m$ in characteristic $p >0$. Certain of their Kummer pullbacks extend to irreducible local systems on the affine line in characteristic $p >0$. One of these, of rank $23$ in characteristic $p=3$, turns out to have the Conway group $\Co_2$, in its irreducible orthogonal representation of degree $23$, as its arithmetic and geometric monodromy groups.
\end{abstract}

\tableofcontents

\section*{Introduction}
In the first two sections, we give the general set up. In the third section, we apply known criteria to show that
certain local systems have finite (arithmetic and geometric) monodromy groups. In the final section, we show that the finite monodromy groups in question are the Conway group $\Co_2$ in its $23$-dimensional irreducible orthogonal representation.

\section{The basic set up, and general results}
We fix a prime number $p$, a prime number $\ell \neq p$, and a nontrivial $\overline{\Q_\ell}^\times$-valued additive character $\psi$ of $\F_p$. For $k/\F_p$ a finite extension, we denote by  $\psi_k$ the nontrivial additive character of $k$ given by $\psi_k :=\psi\circ \Trace_{k/\F_p}$. In perhaps more down to earth terms, we fix a nontrivial $\Q(\mu_p)^\times$-valued additive character $\psi$ of $\F_p$, and a field embedding of
$\Q(\mu_p)$ into  $\overline{\Q_\ell}$ for some $\ell \neq p$.


We fix two integers  $N > D >1$ which are both prime to $p$ and with $\gcd(N,D)=1$. We first describe a rigid local system on $\G_m/\F_p$, denoted $$\sH(\psi, N,D),$$
which is pure of weight $2$ and whose trace function at a point $t \in K^\times$, $K$ a finite extension of $\F_p$, is given by the exponential sum
$$\sum_{x \in K, y \in K^\times}\psi_K(tx^D/y^N -Dx +Ny).$$
It will also turn out that after pullback by $N^{\mathrm {th}}$ power, the pullback system 
$$\sF(\psi, N,D):=[N]^\star\sH(\psi, N,D)$$
extends to an irreducible local system on $\A^1/\F_p$, whose trace function at a point $t \in K$, $K$ a finite extension of $\F_p$, is given by the exponential sum
$$\sum_{x \in K, y \in K^\times}\psi_K(x^D/y^N-Dx + tNy).$$
[In the $\sH$ sum, replace $t$ by $t^N$, and then make the change of variable $y \mapsto ty$, to see what happens over $\G_m$.]

To understand this situation, we must relate $\sH(\psi, N,D)$ to the hypergeometric sheaf

$$\sH yp(\psi,N,D):= \sH yp\biggl(\begin{array}{c}\psi,\mbox{all\ characters\ of\  order\  dividing\ }N;\\ 
    \mbox{all\ nontrivial\ characters\ of\  order\  dividing\ }D\end{array}\biggr).$$
This hypergeometric sheaf is only defined on $\G_m/\F_q$, with $\F_q/\F_p$ an extension large enough to contain all the $ND^{\mathrm {th}}$ roots of unity. We know that it is an irreducible rigid local system on $\G_m$ which is not geometrically isomorphic to any nontrivial multiplicative translate of itself.

In terms of the Kloosterman sheaves
$$\sA:=\sK l(\psi,\mbox{all\ characters\ of\  order\  dividing\ }N)$$
and 
$$\sB:= \sK l(\overline{\psi},\mbox{all\ nontrivial\ characters\ of\  order\  dividing\ }D),$$
we obtain $\sH yp(\psi,N,D)$ as the lower ! multiplicative convolution
$$\sH yp(\psi,N,D)=\sA *_{\times, !} inv^\star \sB.$$

\begin{lem}\label{Nisom}We have a geometric isomorphism
$$\sK l(\psi,\mbox{all\ characters\ of\  order\  dividing\ }N) \cong [N]_\star \sL_{\psi(Nx)}.$$
\end{lem}
\begin{proof}This is \cite[5.6.3]{Ka-GKM}. The twisting factor over extensions $K/\F_p$ containing the $N^{\mathrm {th}}$ roots of unity is
$$A(\psi,N,K) =\prod_{\mbox{\tiny{characters \ }}\rho,\ \rho^N=\triv}(-{\sf{Gauss}}({\psi_K}, \rho)).$$
\end{proof}

The second member, $[N]_\star \sL_{\psi(Nx)}$, makes sense on $\G_m/\F_p$, with trace function given by
$$t \in K^\times \mapsto \sum_{y \in K,\ y^N=t}\psi_K(Ny).$$

\begin{lem}\label{Disom}We have a geometric isomorphism of
$$\sB:=\sK l(\overline{\psi},\mbox{all\ nontrivial\ characters\ of\  order\  dividing\ }D)$$
with the local system $\sB _0$ on $\G_m/\F_p$ whose trace function is
$$t \in K^\times \mapsto - \sum_{x \in K}\psi_K(x^D/t -Dx).$$
\end{lem}
\begin{proof}It suffices to show that over every $K/\Fp$ containing the $D^{\mathrm {th}}$ roots of unity, the two local systems have trace functions
related by 
$$\Trace(Frob_{K,t}|\sB) =\Trace(Frob_{K,t}|\sB_0)\times A(\overline{\psi},D,K)/(\#K),$$
for $A(\overline{\psi},D,K)$ the twisting factor
$$A(\overline{\psi},D,K) =\prod_{\rm{characters \ }\rho, \rho^D=\triv}(-{\sf{Gauss}}(\overline{\psi_K}, \rho)).$$
To show this, it is equivalent to show that their multiplicative Mellin transforms coincide. For $\sB$, the Mellin transform is an explicit product of Gauss sums, cf.\cite [8.2.8]{Ka-ESDE}. For $\sB _0$, using the Hasse-Davenport relation \cite[5.6.1, line -1 on page 84]{Ka-GKM}, we find that the Mellin transform is another product of Gauss sums. The asserted identity is then a straightforward if tedious
calculation using Hasse-Davenport, which we leave to the reader.
\end{proof}

\begin{prop}\label{trace}$\sH yp(\psi,N,D)$ is geometrically isomorphic to the lisse sheaf $$\sA _0 \star_{\times, !} inv^\star \sB _0$$ on $\G_m/\F_p$ whose trace function is that of 
$\sH(\psi, N, D)$.
\end{prop}
\begin{proof}
The trace function of $\sA _0 \star_{\times, !} inv^\star\sB _0$ is minus the multiplicative convolution of the trace functions  of $\sA _0$ and of $inv^\star \sB _0$. Thus
$\sH yp(\psi,N,D)$ is geometrically isomorphic to the lisse sheaf on $\G_m/\F_p$ whose trace function is given by
$$\begin{aligned}u \in K^\times  & \mapsto 
\sum_{s, t \in K,\ st=u}\biggl(\sum_{y \in K,\ y^N=s}\psi_K(Ny)\biggr)\biggl(\sum_{x \in K}\psi_K(x^Dt -Dx)\biggr) \\
& \mbox{(now\ solve\ for\ }s=y^N,t=u/s=u/y^N)\\
& =\sum_{x \in K, y \in K^\times}\psi_K(Ny +ux^D/y^N -Dx).\end{aligned}$$
Because $\sH yp(\psi,N,D)$ is geometrically irreducible, the lisse sheaf $\sH(\psi, N, D)$ is geometrically and hence arithmetically irreducible, and hence is uniquely determined by its trace function. Thus we have an arithmetic isomorphism
$$\sH(\psi,N,D) \cong \sA _0 \star_{\times, !} inv^\star\sB _0.$$
\end{proof}

\begin{lem}\label{weight}The sheaf $\sH(\psi, N, D)$ is pure of weight two. More precisely, we have an arithmetic isomorphism over any extension $K/\F_p$ containing the $ND^{\mathrm {th}}$ roots of unity,
$$\sH(\psi,N,D)\otimes (A(\psi,N,K)A(\overline{\psi},D,K)/(\#K)) \cong \sH yp(\psi, N, D).$$
\end{lem}
\begin{proof}The twisting factor $A(\overline{\psi},D,K)/(\#K)$ has weight $D-3$, and the twisting factor $A(\psi,N,K)$ has weight $N-1$. The hypergeometric sheaf $\sH yp(\psi, N, D)$ is pure of weight $N+D-2$, cf. \cite[ 7.3.8 (5), page 264, and 8.4.13]{Ka-ESDE}.
\end{proof}

\begin{lem}\label{pullback}For any integer $M$ prime to $p$, the pullback
$[M]^\star \sH(\psi, N, D)$ is geometrically irreducible.
The pullback $[N]^\star \sH(\psi, N, D)$ extends to a lisse, geometrically irreducible sheaf on $\A^1/\F_p$, whose trace function is given, at points $u \in K$, $K$ a finite extension of $\F_p$, by
$$u \mapsto \sum_{x \in K, y \in K^\times}\psi_K(x^D/y^N-Dx +uNy).$$
\end{lem}
\begin{proof}To see the asserted geometric irreducibility of $[M]^\star \sH(\psi, N, D)$, we argue as follows. The inner product
$$\langle[M]^\star \sH(\psi, N, D), [M]^\star \sH(\psi, N, D)\rangle =$$
$$=  \langle\sH(\psi, N, D), [M]_\star [M]^\star \sH(\psi, N, D)\rangle.$$
By the projection formula, 
 $$[M]_\star [M]^\star \sH(\psi, N, D) = \sH(\psi, N, D)\otimes [M]_\star \overline{\Q_\ell} =$$
$$ = \sH(\psi, N, D)\otimes \bigoplus_{\chi, \chi^M=\triv}\sL_\chi=$$
$$=\bigoplus_{\chi, \chi^M=\triv} \sL_\chi \otimes  \sH(\psi, N, D).$$
In general, hypergeometric sheaves behave under tensoring with $\sL_\chi$ by
$$\sL_\chi \otimes \sH yp(\psi, \rho_i's;\Lambda_j's)  \cong \sH yp(\psi, \chi \rho_i's;\chi \Lambda_j's).$$
One knows that hypergeometric sheaves are determined geometrically up to multiplicative translation by the sets of their
upstairs and downstairs characters. Because $N$ and $D$ are relatively prime, for any nontrivial $\chi$, the sheaf
$\sL_\chi \otimes  \sH(\psi, N, D)$ will not have the same upstairs and downstairs characters as $ \sH(\psi, N, D)$. Indeed, 
either the upstairs characters change, or $\chi$ is nontrivial of order dividing $N$, in which case the downstairs characters change.

That the $[N]^\star \sH(\psi, N, D)$ pullback is lisse across $0$ is obvious, since the local monodromy at $0$ of $\sH(\psi, N, D)$
is the direct sum of characters of order dividing $N$, so this local monodromy dies after $[N]^\star$. The formula for the
trace results from the formula for the trace of $ \sH(\psi, N, D)$, namely
$$u \mapsto \sum_{x \in K, y \in K^\times}\psi_K(ux^D/y^N-Dx +Ny),$$
by replacing $u$ by $u^N$ and making the substitution $y \mapsto uy$.
\end{proof}

We now view $ \sH(\psi, N, D)$ as a representation of the arithmetic fundamental group $\pi_1^{arith}(\G_m/\F_p):=\pi_1(\G_m/\F_p)$ and of its normal subgroup $\pi_1^{geom}(\G_m/\F_p) :=\pi_1(\G_m/\overline{\F_p})$. And we view the pullback $[N]^\star \sH(\psi, N, D)$ as a representation of $\pi_1^{arith}(\A^1/\F_p):=\pi_1(\A^1/\F_p)$ and of its normal subgroup $\pi_1^{geom}(\A^1/\F_p) :=\pi_1(\A^1/\overline{\F_p})$.

\begin{lem}{\rm (Primitivity Lemma)}\label{primitivity}
Suppose that $N > D >1$  are both prime to $p$ and have $\gcd(N,D)=1$. Then we have the following results.
\begin{enumerate}[\rm(i)]
\item Unless $D$ is $3$ or $4$ or $6$, the local system $[N]^\star \sH(\psi, N, D)$ on $\A^1/\F_p$ is not geometrically induced, i.e., there is no 
triple $(U,\pi, \sG)$ consisting of a
connected smooth curve $U/\overline{\F_p}$, a finite etale map $f:U \rightarrow \A^1/\overline{\F_p}$ of degree $d \ge 2$, and a local system $\sH$ on $U$ such that there exists an isomorphism of $\pi_\star \sG$ with (the pullback to $\A^1/\overline{\F_p}$ of) $[N]^\star \sH(\psi, N, D)$.
\item Unless $D$ is $3$ or $4$ or $6$, the local system $ \sH(\psi, N, D)$ on $\G_m/\F_p$ is not geometrically induced.
\item Suppose $D=3$. Then  $[N]^\star \sH(\psi, N, D)$ is not geometrically induced unless $N=1+q$ for some power $q$ of $p$.
In this case,  $\sH(\psi,N=1+q, D=3)$ is geometrically induced, and hence so is $[N]^\star \sH(\psi, N, D)$.
\item Suppose $D=4$. Then  $[N]^\star \sH(\psi, N, D)$ is not geometrically induced unless $N=1+2q$  or $N=2+q$ for some power $q$ of $p$. In this case, both $\sH(\psi,N=1+2q, D=4)$ and  $\sH(\psi,N=2+q, D=4)$ are geometrically induced, and hence so are both
$[N]^\star \sH(\psi, N=1+2q, D=4)$ and $[N]^\star \sH(\psi, N=2+q, D=4)$.
\item Suppose $D=6$. Then  $[N]^\star \sH(\psi, N, D)$ is not geometrically induced (and hence $ \sH(\psi, N, D)$ is not geometrically induced).
\end{enumerate}
\end{lem}

\begin{proof}For (i), we argue as follows. The pullback sheaf  $[N]^\star \sH(\psi, N, D)$  has Euler characteristic zero, as its rank, $N$, is equal to its Swan conductor.
So if such a triple $(U,\pi, \sG)$ exists, we have the equality of Euler characteristics
$$EP(U,\sG) =EP(\A^1/\overline{\F_p},\pi_\star \sG) =EP(\A^1/\overline{\F_p},[N]^\star \sH(\psi, N, D))=0.$$
Denote by $X$ the complete nonsingular model of $U$, and by $g_X$ its genus. Then $\pi$ extends to a finite flat map of $X$ to $\P^1$, and the Euler-Poincar\'{e} formula gives
$$0=EP(U,\sG) = \rank(\sG)(2-2g_X -\#(\pi^{-1}(\infty) )) - \sum_{w \in \pi^{-1} - (\infty)}\Swan_w(\sG).$$
Thus $g_X=0$, otherwise already the first term alone is strictly negative. So now $X=\P^1$, and we have
$$0 =  \rank(\sG)(2 -\#(\pi^{-1}(\infty) )) - \sum_{w \in \pi^{-1} - (\infty)}\Swan_w(\sG).$$

If $\#(\pi^{-1}(\infty)) \ge 3$, then already the first term alone is strictly negative.
If $\#(\pi^{-1}(\infty)) =1$, then $U$ is $\P^1 \setminus ({\rm one \ point}) \cong \A^1$, and so $\pi$ is a finite etale map of $\A^1$ to itself of degree $> 1$. But any such map has degree divisible by $p$, and hence $\pi_\star \sG$ would have rank divisible by $p$. But its rank is $N$, which is prime to $p$. Thus we must have $\#(\pi^{-1}(\infty)) =2$ (and $g_X=0$). 

Throwing the two points to $0$ and $\infty$, we have a finite etale
map
$$\pi:\G_m \rightarrow \A^1.$$
The equality of EP's now gives 
$$0 =\Swan_0(\sG) + \Swan_\infty(\sG).$$
Thus $\sH$ is lisse on $\G_m$ and everywhere tame, so a successive extension of lisse, everywhere tame sheaves of rank one.
But $\pi_\star \sH$ is irreducible, so $\sH$ must itself be irreducible, hence of rank one, and either $\overline{\Q_\ell}$ or an $\sL_{\rho}$.
[It cannot be $\overline{\Q_\ell}$, because $\pi_\star \overline{\Q_\ell}$ is not irreducible when $\pi$ has degree $>1$; by adjunction $\pi_\star \overline{\Q_\ell}$ contains $\overline{\Q_\ell}$.]

Now consider the maps induced by $\pi$ on punctured formal neighborhoods
$$\pi(0): \G_m{(0)} {\rightarrow  \A^1}{(\infty)},\ \ \pi(\infty): \G_m{(\infty)} \rightarrow  {\A^1}{(\infty)}.$$
The $I(\infty)$-representation of $\sF_{p,D,f,\chi}$ is then the direct sum
$$\pi(0)_\star \sL_{\rho} \oplus \pi(\infty)_\star \sL_{\rho}.$$

Denote by $d_0$ and $d_\infty$ their degrees. Because $\sG$ has rank one, the degree of $\pi$ must be $N$, and hence
$$N=d_0 + d_\infty.$$
Both $d_0$ and $d_\infty$ cannot be prime to $p$, for then the $I(\infty)$ representation would be the sum of $d_0+d_\infty =N$ tame characters. But the  $I(\infty)$ representation of $[N]^\star \sH(\psi, N, D)$ is the direct sum of the $D-1 <N$ nontrivial characters of order dividing $D$ and a wild part of rank $N -(D-1)$.

After interchanging $0$ and $\infty$ if necessary, we may assume that $d_0$ is prime to $p$, and that 
$$d_1 = n_1q$$
with $n_1$ prime to $p$ and with $q$ a positive power of $p$.
Then $\pi(0)_\star \sL_{\rho}$ consists of $d_0$ distinct characters, the $d_0^{\mathrm {th}}$ roots of $\sL_{\rho}$, while 
$\pi(\infty)_\star \sL_{\rho}$ consists of $n_1$ distinct characters (the $n_1q^{\mathrm {th}}$ roots of $\sL_{\rho}$), together with a wild part.
As the total number of tame characters is $D-1$, we have the equalities
$$N=d_0 + n_1q, \ \ D-1 =d_0 +n_1.$$
There is now a further observation to be made. The $d_0$ tame characters occuring in  $\pi(0)_\star \sL_{\rho}$ are a torsor under the characters of order dividing $d_0$. So the ratio of any two has order dividing $d_0$. But each of these $d_0$ characters has order dividing $D$, and hence so does any ratio of them. Therefore $d_0$ divides $D$. Similarly, we see that $n_1$ divides $D$.

Thus we have $D-1 = d_0 + n_1$, with both $d_0$ and $n_1$ divisors of $D$. We will see that this is very restrictive. Write
$$d_0=D/A, \ n_1=D/B,$$
with $A,B$ each divisors of $D$.

We first observe that both $A,B$ must be $\ge 2$, otherwise one of the terms $D/A$ or $D/B$ is $D$, already too large.

Suppose first $D$ is even. We cannot have $A=B=2$, because then $D/A +D/B =D$ is too large. So the largest possible value of 
$D/A +D/B$ is $D/2 +D/3 =5D/6$. This is too small, i.e. $5D/6 < D-1$, so long as $5D <6D-6$, i.e., so long as $D > 6$. And we note that for $D=6$, we indeed have $D-1=5 = 2+3$ is the sum of divisors of $D$.

Suppose next that $D$ is odd. Then the largest  possible value of 
$D/A +D/B$ is $D/3 +D/3 =2D/3$. This is too small, i.e., $2D/3 < D-1$, so long as $D>3$. Notice that $D-1=2 = 1+1$ is the sum of divisors of $D$. 

The case $D=2$ cannot be induced, because
the lowest value for $D-1=1$ is $1+1$ (coming from $N=1+q$), too large.

This concludes (!) the proof of (i).

Assertion (ii) follows trivially, since if a representation of the group is primitive, it is all the more primitive on an overgroup.

For assertion (iii), the case $D=3$, we can only achieve $D-1=2$ as as the sum of two divisors of $D=3$ as $D-1 =2 =1+1$, i.e. if $N=1+q$. In this case, $ q$ must be $1$ mod $3$ (simply because $N=q+1$ is prime to $D=3$, as is $p$ and hence also $q$). In this case, one checks that the
finite etale map $\pi:x \mapsto 1/(x^q(x-1))$ from $\G_m \setminus \{1\}$ to $\G_m$ has 
$$\pi_\star(\sL{\chi_3(x)}\otimes \sL{\chi_3^2(x)}),$$
for $\chi_3$ either character of order $3$, geometrically isomorphic to a multiplicative translate of $\sH yp(\psi,N,D)$.

For assertion (iv), the case $D=4$, we can only achieve $D-1=3$ as $1+2$, i.e. as $N=1+2q$ or as $N=2+q$. Here $q$ must be odd, as $p$ is prime to $D=4$. One checks that for both of the finite etale maps $\pi:x \mapsto 1/(x^q(x-1)^2)$ and $\pi:x \mapsto 1/(x^2(x-1)^q)$, 
$$\pi_\star\sL{\chi_2(x(x-1))},$$
for $\chi_2$ the quadratic character, is geometrically isomorphic to a multiplicative translate of $\sH yp(\psi,N,D)$.

For assertion (v), the case $D=6$, we argue by contradiction. The two possibilities are $N=2+3q$ and $N=3+2q$. In both cases, the tame characters at $\infty$ will be, for some nontrivial character $\rho$, the union of the square roots of either $\rho$ or of $\rho^q$ and the cube roots of either $\rho$ or of $\rho^q$. These are to be the nontrivial characters of order dividing $D=6$. Thus $\rho$ is the cube of some character of order dividing $6$, but is nontrivial, so $\rho$ must be the quadratic character. But $\rho$ is also the square of some character of order dividing $6$, but being nontrivial must have order $3$. So this $D=6$ case cannot be induced. A fortiori, in this $D=6$ case $\sH yp(\psi,N, D)$ cannot be induced, cf. the proof of (ii).
\end{proof}

\begin{lem}\label{det}\rm{(Determinant Lemma)} Suppose that $N > D >1$  are both prime to $p$ and have $\gcd(N,D)=1$.  Then we have the following results.  
\begin{enumerate}[\rm(i)]
\item If $N$ is odd, then $\det(\sH(\psi,N,D)(1))$ is geometrically constant, It is of the form $A^{deg}$, with $A =\zeta$ for some root of unity $\zeta \in \Z[\zeta_p]$.
\item if $N$ is even, then  $\det([N]^\star \sH(\psi,N,D)(1))$ is geometrically constant. It is of the form $A^{deg}$, with $A =\zeta$ for some root of unity $\zeta \in \Z[\zeta_p]$.
\item Suppose $N$ is odd (respectively even) and we have the congruence 
$$D \equiv N+1 (\bmod\ p-1).$$
Then $\sH(\psi,N,D)(1)$ (respectively $[N]^\star \sH(\psi,N,D)(1)$) has all Frobenius traces in $\Q$, and
has determinant $A^{deg}$ with $A$ some choice of $\pm  1$.
\item Under the hypotheses of (iii), denote by $d$ the degree of $\F_p(\mu_N)/\F_p$. 
If $N$ is odd, or if $N$ is even and either $D$ is a square in $\F_p$ or $d$ is even, 
then $A^d =1$. If $N$ is even, and $D$ is a nonsquare in $\F_{p^d}$, then $A^d=-1.$
In particular, if  $N$ and $d$ are both odd, then
we have $A=1$ in (iii) above.
\end{enumerate}                                                                                                                                                                                               
\end{lem}
\begin{proof}To prove (i) and (ii), we appeal to \cite[8.11.6]{Ka-ESDE}. Here the upper characters are all the characters of order dividing $N$, so their product is trivial when $N$ is odd, and the quadratic character $\chi_2$ when $N$ is even. Thus the determinant is geometrically trivial when $N$ is odd, and is geometrically $\sL_{\chi_2}$ when $N$ is even. In the $N$ even case, it becomes geometrically trivial after $[N]^\star$ (or just after $[2]^\star$). So we can compute $A$ as the determinant at, say, the point $s=1$ (which is certainly an $N^{\mathrm {th}}$ power). 

This determinant lies in $\Z[\zeta_p][1/p]$, and is pure of weight zero, i.e., it has absolute value $1$ for every complex embedding of $\Q(\zeta_p)$. Because $\Q(\zeta_p)$ has a unique place over $p$, this determinant and its complex conjugate have the same $p$-adic ord. So being of weight zero, the determinant is a unit
at the unique place over $p$. As it lies in $\Z[\zeta_p][1/p]$, it is integral at all $\ell$-adic places over primes $\ell \neq p$. So by the product formula, it must be a unit everywhere, and so is a root of unity in $\Q(\zeta_p)$.

To prove (iii), we first show that under the asserted congruence, each sum
$$\sum_{x \in K, y \in K^\times}\psi_K(tx^D/y^N -Dx +Ny).$$
lies in $\Z$. Indeed, this sum lies in $\Z[\zeta_p]$, so it suffices to show that it is invariant under the Galois group $\Gal(\Q(\zeta_p)/\Q)$.
This group is $\F_p^\times$, with $\alpha \in \F_p^\times$ moving this sum to
$$\sum_{x \in K, y \in K^\times}\psi_K(\alpha tx^D/y^N -D\alpha x +N \alpha y)=$$
$$=\sum_{x \in K, y \in K^\times}\psi_K( t(\alpha x)^D/(\alpha y)^N -D\alpha x +N \alpha y),$$
(equality because $\alpha =\alpha^{D-N}$ by the congruence $D \equiv N+1 {\rm\ mod \ }p-1$)
which is the original sum after the change of variable $x \mapsto \alpha x, \ y \mapsto \alpha y$.
Once the traces lie in $\Q$, the same argument used in (ii) above now shows that the determinant is a root of unity in $\Q$.

To prove (iv), we argue as follows. The assertion is about the sign of $A^d$, which we can compute for $[N]^\star \sH(\psi,N,D)(1)$ at any point $t \in \F_{p^d}$. We take the point $t=0$. We will compute the trace over all extensions $K=\F_q$ of $ \F_{p^d}$ in order to calculate the eigenvalues of $Frob_{t=0, \F_{p^d}}$ on $[N]^\star \sH(\psi,N,D)(1)$. Over such a $K$, the trace is
$$\sum_{x \in K, y \in K^\times}\psi_K(x^D/y^N -Dx )/q.$$
In {\bf{this}} sum, we may invert $y$, so the sum becomes 
$$(1/q)\sum_{x \in K}\psi_K(-Dx)\sum_{y \in K^\times}\psi_K(x^Dy^N).$$
Because all the characters of order dividing $N$ exist over $K$, we rewrite the second term as
$$\sum_{y \in K^\times}\psi_K(x^Dy^N) =\sum_{u \in K^\times}\psi_K(x^Du) \sum_{\rm{char.'s\ } \rho, \rho^N=\triv}\rho(u).$$
So our $K$ sum becomes
$$(1/q)\sum_{x \in K}\psi_K(-Dx) \sum_{\rm{char.'s\ } \rho, \rho^N=\triv}\sum_{u \in K^\times}\psi_K(x^Du)\rho(u).$$
For each nontrivial $\rho$, the $\rho$ term is 
$$\sum_{u \in K^\times}\psi_K(x^Du)\rho(u) = \overline{\rho}(x^D) {\sf{Gauss}}(\psi_K, \rho).$$
So each $\rho \neq \triv$ term in the $K$ sum is 
$$\sum_{x \in K}\psi_K(-Dx) \overline{\rho}(x^D){\sf{Gauss}}(\psi_K, \rho)/q ={\rho}^D(D){\sf{Gauss}}(\overline{\psi_K},\overline{\rho}^D){\sf{Gauss}}(\psi_K, \rho)/q.$$
The $\rho=\triv$ term in the $K$ sum is 
$$(1/q)\sum_{x \in K}\psi_K(-Dx)\sum_{u \in K^\times}\psi_K(x^Du)=$$
$$= (1/q)\sum_{x \in K}\psi_K(-Dx)(-1 + \sum_{u \in K}\psi_K(x^Du))=$$
$$= (1/q)\sum_{x \in K^\times}\psi_K(-Dx)(-1) + (1/q)(-1 +q)= (1/q)(1+q-1) = 1.$$
Thus the eigenvalues are precisely $1$ and, for each of the $N-1$ nontrivial $\rho$ of order dividing $N$, the product
$${\rho}^D(D){\sf{Gauss}}(\overline{\psi_{ \F_{p^d}}},\overline{\rho}^D){\sf{Gauss}}(\psi_{ \F_{p^d}}, \rho)/p^d.$$
So the determinant is the product 
$$\prod_{\rho^N=\triv, \ \rho \neq \triv}{\rho}^D(D)[{\sf{Gauss}}(\overline{\psi_{ \F_{p^d}}},\overline{\rho}^D){\sf{Gauss}}(\psi_{ \F_{p^d}}, \rho)/p^d].$$
Because $D$ is prime to $N$, the $\rho^D$ are just a rearrangement of the $\rho$. So defining 
$$\Lambda:= \prod_{\rho^N=\triv, \ \rho \neq \triv}\rho,$$
we see that the determinant is
$$\Lambda(D)\prod_{\rho^N=\triv, \ \rho \neq \triv}[{\sf{Gauss}}(\overline{\psi_{ \F_{p^d}}},\overline{\rho}){\sf{Gauss}}(\psi_{ \F_{p^d}}, \rho)/p^d] = \Lambda(D).$$
If $N$ is odd, then $\Lambda =\triv$. If $N$ is even, the $\Lambda$ is the quadratic character of $\F_{p^d}$. In this case, $ \Lambda(D)=1$ if either $D$ is already a square in $\F_p$, or if $d$ is even, so that $D$ becomes a square in $\F_{p^d}$.
\end{proof}

\section{The criterion for finite monodromy}
Denote by $V$ Kubert's $V$-function
$$V: (\Q/\Z)_{{\rm prime\  to \ }p} \rightarrow \Q_{\ge 0}.$$
It has the following property. For $f \ge 1$, $q:=p^f$, $\Teich_f: \F_q^\times \rightarrow \Q_p(\mu_{q-1})$ the Teichmuller character (for a fixed $p$-adic place of $\Q_p(\mu_{q -1})$), and $x \in (\Q/\Z)_{{\rm prime\  to \ }p}$ of order dividing $q-1$, we have
$$V(x)=\ord_{p^f}({\sf{Gauss}}(\psi_{\F_q},\Teich_f^{-(q-1)x})).$$

\begin{lem}\label{integrality}Suppose that $N > D >1$  are both prime to $p$ and have $\gcd(N,D)=1$.  Then $\sH(\psi,N,D)(-1)$ has algebraic integer traces, and hence finite arithmetic monodromy group $G_{arith}$, if and only if the following inequalities hold.
For every $x \in (\Q/\Z)_{{\rm prime\  to \ }p}$, we have
$$V(Nx)+V(-Dx)-V(-x) \ge 0.$$
Equivalently (since this trivially holds for $x=0$), the condition is that for every nonzero  $x \in (\Q/\Z)_{{\rm prime\  to \ }p}$, we have
$$V(Nx)+V(-Dx)+V(x) \ge 1.$$
\end{lem}
\begin{proof} From Lemma \ref{weight}, we see that $\sH(\psi,N,D)(-1)$ has algebraic integer traces if and only if, for every finite extension $K/\F_p$ containing the $ND^{\mathrm {th}}$ roots of unity, the hypergeometric sheaf $\sH yp(\psi,N,D)$ has traces at $K$-points which are of the form $A(N,K)A(D,K)$ times algebraic integers. Equivalently, as explained in \cite{Ka-RL-T}, it suffices that the Mellin transform
of the trace function of $\sH yp(\psi,N,D)$ on $K^\times$ has all values with ord at least that of $A(N,K)A(D,K)$. The value at $\chi$ is
$$(-1)^{N-D}\biggl(\prod_{\rho,\ \rho^N=\triv}\Gauss(\psi_K,\chi\rho)\biggr)\times
\biggl(\prod_{\sigma, \sigma^D=\triv,\sigma \neq \triv}{\sf{Gauss}}(\overline{\psi}_K,\overline{\chi\rho})\biggr).$$
The first product is, up to a root of unity factor,
$$A(N,K)\Gauss(\psi_K,\chi^N).$$ 
The second product is, up to a root of unity factor,
$$A(D,K)\Gauss(\overline{\psi}_K,\overline{\chi}^D)/\Gauss(\overline{\psi}_K,\overline{\chi}).$$
So the requirement is that for all $\chi$, the product
$$\Gauss(\psi_K,\chi^N)\Gauss(\overline{\psi}_K,\overline{\chi}^D)/\Gauss(\overline{\psi}_K,\overline{\chi})$$
have nonnegative ord. This is precisely the 
$$V(Nx)+V(-Dx)-V(-x) \ge 0$$
version of the criterion. 
\end{proof}

\begin{rmk}\label{sawin}As Will Sawin pointed out to us, there are a number of sheaves $\sH yp(\psi, N, D)$ which, although not induced,
are very nearly so. Namely, for any power $q$ of $p$, and for any $D \ge 2$ prime to $p$, the direct image $\pi_\star \overline{\Q_\ell}$, for $\pi:\G_m \setminus \{1\} \rightarrow \G_m$ the finite \'{e}tale map given by $x \mapsto x^{Dq-1}(x-1)$, 
is geometrically isomorphic to the direct sum of the constant sheaf $\overline{\Q_\ell}$ and $\sH yp(\psi, Dq-1, D)$. The case $D=2$ was the subject of the paper \cite{G-K-T}.

From this direct image picture, we see that $\sH yp(\psi, Dq-1, D)$ has finite geometric monodromy, and hence that $\sH(\psi, Dq-1, D)(1)$
has finite $G_{arith}$. In other words, the inequality of Lemma \ref{integrality} is satisfied by the data $N=Dq-1, D$ in the characteristic $p$ of which $q$ is a power and to which $D$ is prime. 

The induced cases, being induced from rank one, also have finite geometric monodromy. Thus $\sH(\psi, q+1, 3)(1)$ has finite $G_{arith}$ for any prime power $q$ which is $1$ mod $3$ in the characteristic of which $q$ is a power. And both $\sH(\psi, q+2, 4)(1)$ and $\sH(\psi, 2q+1, 4)(1)$  have finite $G_{arith}$ in the odd characteristic of which $q$ is a power. This gives other cases of data satisfying the inequality of Lemma \ref{integrality}.

In the next section, we will exhibit another datum, {\bf not} of either of these types, satisfying the
inequality. How many others are there?
\end{rmk}

\section{Theorems of finite monodromy}
In this section, we will prove

\begin{thm}\label{mainfinite} For $p=3$, $N=23$ and $D=4$, the sheaf $\sH(\psi,N,D)(1)$ has finite arithmetic and geometric monodromy groups.
\end{thm}

By lemma \ref{integrality}, we need to show that for every nonzero  $x \in (\Q/\Z)_{{\rm prime\  to \ }3}$, we have
$$V(23x)+V(-4x)-V(-x) \ge 0$$
or, equivalently,
$$V(-23x)+V(4x)-V(x) \ge 0.$$
Applying the duplication formula $V(x)+V(x+\frac{1}{2})=V(2x)+\frac{1}{2}$, see \cite[p.206]{Ka-G2hyper}, twice, this is equivalent to
$$
1-V(-23x)\leq V\left(2x+\frac{1}{2}\right)+V\left(x+\frac{1}{2}\right).
$$
For $x=\frac{1}{2}$ this is obvious so, making the change of variable $x\mapsto x+\frac{1}{2}$, we get the equivalent condition
$$
1-V\left(-23x-\frac{1}{2}\right)\leq V\left(2x+\frac{1}{2}\right)+V(x).
$$
If $x=\frac{1}{4}$ or $x=\frac{3}{4}$ the inequality (equality in this case) is trivial. Otherwise, $2x+\frac{1}{2}\neq 0$, and we can rewrite it as
$$
1-V\left(-23x-\frac{1}{2}\right)\leq 1-V\left(-2x-\frac{1}{2}\right)+V(x).
$$

We now restate the inequality in terms of the function $\lbr -\rbr_r:=\lbr -\rbr_{3,r}$ defined in \cite{R-L}, given by 
$$\begin{aligned}\ \lbr x\rbr_{3,r}= & \mbox{ the sum of the 3-adic digits of the representative of}\\
    & \mbox{ the congruence class of }x\mbox{ modulo }3^r-1\mbox{ in }[1,3^r-1].\end{aligned}$$ 
By \cite[Appendix]{Ka-RL} we have $\lbr x\rbr_r=2r(1-V(-\frac{x}{3^r-1}))$, so the finite monodromy condition can be restated as
$$
\left[23x+\frac{3^r-1}{2}\right]_r\leq\lbr x\rbr_r+\left[ 2x+\frac{3^r-1}{2}\right]_r
$$
for every $r\geq 1$ and every integer $0<x<3^r-1$.

For a non-negative integer $x$, let $\lbr x\rbr $ denote the sum of the $3$-adic digits of $x$. For use below, we recall the following result from \cite[Prop. 2.2]{Ka-RL}:

\begin{prop}\label{inequality}
For strictly positive integers $x$ and $y$, and any $r \ge 1$, we have:
\begin{enumerate}[\rm(i)]
 \item $\lbr x+y\rbr \leq \lbr x\rbr +\lbr y\rbr $;
 \item $\lbr x\rbr _{r} \le \lbr x\rbr $;
 \item $\lbr 3x\rbr =\lbr x\rbr $.
\end{enumerate}
\end{prop}

\begin{lem}
 Let $r\geq 1$ and $0\leq x <3^r$ an integer. Then
 $$
\left[23x+\frac{3^r-1}{2}\right]\leq\lbr x]+\left[2x+\frac{3^r-1}{2}\right]+2.
$$
Moreover, if $r=1$ and $x\neq 1$,  $r=2$ and $x\neq 3$ or $r\geq 3$ and the first three $3$-adic digits of $x$ (after adding leading 0's so it has exactly $r$ digits) are not $100$ or $202$, then
 $$
\left[23x+\frac{3^r-1}{2}\right]\leq\lbr x\rbr +\left[2x+\frac{3^r-1}{2}\right].
$$
\end{lem}

\begin{proof}
 We proceed by induction on $r$: for $r\leq 3$ one checks it by hand. Let $r\geq 4$ and $0\leq x< 3^r$. 
 
\smallskip
{\it Case 1:} $x\equiv 0$ (mod 3). 
 
 Write $x=3y$ with $0\leq y <3^{r-1}$. Then
 $$\begin{aligned}
 \left[23x+\frac{3^r-1}{2}\right] & =\left[3\left(23y+\frac{3^{r-1}-1}{2}\right)+1\right]\\
 & =\left[23y+\frac{3^{r-1}-1}{2}\right]+1\\
 & \leq \lbr y\rbr +\left[2y+\frac{3^{r-1}-1}{2}\right]+3\\
 & = \lbr y\rbr +\left[3\left(2y+\frac{3^{r-1}-1}{2}\right)+1\right]+2\\
 & =\lbr x\rbr +\left[2x+\frac{3^{r}-1}{2}\right]+2 \end{aligned}
 $$
 by induction. Since the first three digits of $x$ and $y$ are the same, the better inequality holds when those three digits are not $100$ or $202$, also by induction.
 
\smallskip
{\it Case 2:} The last $r-3$ digits of $x$ contain the string $00$ or the string $01$.

Say the string is located at position $s\leq r-3$, counting from the right. Write $x=3^sy+z$, where $y<3^{r-s}$ and $z<2\cdot 3^{s-2}$. Then 
$$2z+\frac{3^s-1}{2}<4\cdot 3^{s-2}+\frac{3^s-1}{2}<3^s\bigl(\frac{4}{9}+\frac{1}{2}\bigr)<3^s,$$ 
so
$$\begin{aligned}
\left[2x+\frac{3^r-1}{2}\right] & =\left[3^s\left(2y+\frac{3^{r-s}-1}{2}\right)+2z+\frac{3^s-1}{2}\right]\\
& = \left[2y+\frac{3^{r-s}-1}{2}\right]+\left[2z+\frac{3^{s}-1}{2}\right],\end{aligned}$$
and therefore
$$\begin{aligned}
\left[23x+\frac{3^r-1}{2}\right] & =\left[3^s\left(23y+\frac{3^{r-s}-1}{2}\right)+23z+\frac{3^s-1}{2}\right]\\
& \leq\left[23y+\frac{3^{r-s}-1}{2}\right]+\left[23z+\frac{3^{s}-1}{2}\right]\\
& \leq \lbr y\rbr +\left[2y+\frac{3^{r-s}-1}{2}\right]+2+ \lbr  z\rbr +\left[2z+\frac{3^{s}-1}{2}\right] \\
& =[x\rbr +\left[2x+\frac{3^{r}-1}{2}\right] +2\end{aligned}$$
by induction. Again, since the first three digits of $x$ and $y$ are the same, the better inequality holds if they are not $100$ or $202$.

\smallskip
{\it Case 3:} The last $r-3$ digits of $x$ contain one of the strings $02$, $10$, $11$ or $12$, and the previous digit is not a $2$.

Say the string is located at position $s\leq r-3$, counting from the right. Write $x=3^sy+z$, where $y<3^{r-s}$ is not $\equiv 2$ mod $3$ and $z<2\cdot 3^{s-1}$. Then 
$$2z+\frac{3^s-1}{2}<4\cdot 3^{s-1}+\frac{3^s-1}{2}<3^s\bigl(\frac{4}{3}+\frac{1}{2}\bigr)<2\cdot 3^s$$ 
and the last digit of $2y+(3^{r-s}-1)/2$ is not $2$, so
$$\begin{aligned}
\left[2x+\frac{3^r-1}{2}\right]  & =\left[3^s\left(2y+\frac{3^{r-s}-1}{2}\right)+2z+\frac{3^s-1}{2}\right] \\
& =\left[2y+\frac{3^{r-s}-1}{2}\right] +\left[2z+\frac{3^{s}-1}{2}\right] .\end{aligned}$$
We conclude as in case $1$ {\bf unless} $s=2$ and $z=3$ (in which case we can apply case 1) or $s\geq 3$ and the first three digits of $z$ are $100$ (in which case we can apply case 2).

\medskip
If $x$ is not included in the cases proved so far and the last $r-3$ digits of $x$ contain a $0$, it must be enclosed between two $2$'s. If they contain a $1$, it must be enclosed between two $2$'s, with the possible exception of the last two digits in the case when $x$ ends with $211$ or $21$.

\smallskip
{\it Case 4:} The last $r-3$ digits of $x$ contain a $2$ which is preceded by a $1$ and the next two digits (if they exist) are not $02$.

Write $x=3^sy+z$, where $y<3^{r-s}$ is congruent to $1$ mod $3$ and $z< 3^s$ starts with $2$ (but not with $202$). Then 
$$2z+\frac{3^s-1}{2}<2\cdot 3^{s}+\frac{3^s-1}{2}<3^{s+1}$$ 
and the last digit of $2y+(3^{r-s}-1)/2$ is $0$, so
$$\begin{aligned}
\left[2x+\frac{3^r-1}{2}\right]  & =\left[3^s\left(2y+\frac{3^{r-s}-1}{2}\right)+2z+\frac{3^s-1}{2}\right] \\
& =\left[2y+\frac{3^{r-s}-1}{2}\right] +\left[2z+\frac{3^{s}-1}{2}\right] .\end{aligned}
$$
We conclude as in the previous two cases.

\smallskip
{\it Case 5:} A digit of $x$ which is not one of the first four or the last two is a $1$. By the note after case 3, we can assume that the $1$ is enclosed bewteen two $2$'s. If the digit after the following $2$ is not $0$ we apply case 4. If it is $0$ and is the last digit, we apply case 1. Otherwise, by case $4$ we can assume that there is a $2$ after the $0$, so the last $r-3$ digits of $x$ contain the string $21202$. If the previous digit is $1$ we apply case 4.

Otherwise, write $x=3^sy+z$, where the last three digits of $y$ are $021$ or $221$, and the first three digits of $z$ are $202$. Then the last $3$ digits of $2y+(3^{r-s}-1)/2$ are $000$ or $100$ and $2z+(3^s-1)/2<3^{s+2}$, so
$$\begin{aligned}
\left[2x+\frac{3^r-1}{2}\right]  & =\left[3^s\left(2y+\frac{3^{r-s}-1}{2}\right)+2z+\frac{3^s-1}{2}\right] \\
& =\left[2y+\frac{3^{r-s}-1}{2}\right] +\left[2z+\frac{3^{s}-1}{2}\right] .\end{aligned}
$$
On the other hand, the last three digits of $23y+(3^{r-s}-1)/2$ are $110$ or $210$, and $23=212_3$. Since $212_3\cdot 202_3=122011_3$, $23z+(3^{s}-1)/2$ has exactly $s+3$ digits, the first three of them being at least $122$. In any case, in the sum 
$3^s\bigl(23y+(3^{r-s}-1)/2\bigr)+\bigl(23z+(3^{s}-1)/2\bigr)$ there is at least one digit carry, so
$$\begin{aligned}
\left[23x+\frac{3^r-1}{2}\right]  & =\left[3^s\left(23y+\frac{3^{r-s}-1}{2}\right)+23z+\frac{3^s-1}{2}\right] \\
& \leq\left[23y+\frac{3^{r-s}-1}{2}\right] +\left[23z+\frac{3^{s}-1}{2}\right] -2\\
& \leq \lbr y\rbr +\left[2y+\frac{3^{r-s}-1}{2}\right] +2+ \lbr  z\rbr +\left[2z+\frac{3^{s}-1}{2}\right] \\
& =[x\rbr +\left[2x+\frac{3^{r}-1}{2}\right] +2.\end{aligned}
$$
As usual, if the first three digits of $x$ (or $y$) are not $100$ or $202$ one gets the better bound.

\smallskip
{\it Case 6:} The last $r-1$ digits of $x$ contain the string $2202$. Write $x=3^sy+z$, where $y<3^{r-s}$ ends with $22$ and $z<3^s$ starts with $02$. Then the last two digits of $2y+(3^{r-s}-1)/2$ are $02$, and $2z+(3^{s}-1)/2$ has at most $s+1$ digits. So
$$\begin{aligned}
\left[2x+\frac{3^r-1}{2}\right]  & =\left[3^s\left(2y+\frac{3^{r-s}-1}{2}\right)+2z+\frac{3^s-1}{2}\right] \\
& \geq\left[2y+\frac{3^{r-s}-1}{2}\right] +\left[2z+\frac{3^{s}-1}{2}\right] -2.\end{aligned}
$$
On the other hand, the last two digits of $23y+(3^{r-s}-1)/2$ are $22$, and $23z+(3^{s}-1)/2$ has $s+2$ digits, the first two being at least $12$. In any case, in the sum $3^s\bigl(23y+(3^{r-s}-1)/2\bigr)+\bigl(23z+(3^{s}-1)/2\bigr)$ there is at least one digit carry, so
$$\begin{aligned}
\left[23x+\frac{3^r-1}{2}\right]  & =\left[3^s\left(23y+\frac{3^{r-s}-1}{2}\right)+23z+\frac{3^s-1}{2}\right] \\
& \leq\left[23y+\frac{3^{r-s}-1}{2}\right] +\left[23z+\frac{3^{s}-1}{2}\right] -2\\
& \leq \lbr y\rbr +\left[2y+\frac{3^{r-s}-1}{2}\right] +2+ \lbr  z\rbr +\left[2z+\frac{3^{s}-1}{2}\right] -2\\
& \leq \lbr x\rbr +\left[2x+\frac{3^{r}-1}{2}\right] +2.\end{aligned}
$$
As usual, if the first three digits of $x$ (or $y$) are not $100$ or $202$ one gets the better bound.

\smallskip
{\it Case 7:} The last $r-1$ digits of $x$ contain the string $2222$. This case is similar to the previous one, with the difference that 
$23z+(3^s-1)/2$ has now $s+3$ digits, the first three being $202$, $210$, $211$ or $212$.

\medskip
For all remaining cases, all digits except for the first four and the last two must be $0$'s and $2$'s. Among them, there can not be two consecutive $0$'s or four consecutive $2$'s, and if there is a string of two or three consecutive $2$'s, it can not be followed by a $0$. So the last $r-4$ digits of $x$ are of the form
$$\begin{aligned}
 \mbox{(possibly a }0) + (a\mbox{ copies of }20)\ + \\
 (1, 2 \mbox{ or }3\ 2\mbox{'s}) + \mbox{(possibly 1 or 2 }1\mbox{'s})
\end{aligned}$$
for some $a\geq 0$. Here is a table with the last digits of $2x+(3^r-1)/2$ and $23x+(3^r-1)/2$ in each case for $a\geq 2$:
\bigskip

\begin{tabular}{c|c|c}
 last digits of $x$ & last digits of $2x+\frac{3^r-1}{2}$ & last digits of $23x+\frac{3^r-1}{2}$  \\ \hline\hline
 
 $\overbrace{2020...20}^{a\times 20}2$ & $\overbrace{22...2}^{2a\times 2}2$ & $\overbrace{2020...20}^{(a-1)\times 20}112$ \\
 \hline
 $\overbrace{2020...20}^{a\times 20}22$ & $\overbrace{00...0}^{2a\times 0}02$ & $\overbrace{2020...20}^{(a-2)\times 20}210022$ \\
 \hline
 $\overbrace{2020...20}^{a\times 20}222$ & $\overbrace{00...0}^{2a\times 0}102$ & $\overbrace{2020...20}^{(a-2)\times 20}2102122$ \\
 \hline
 $\overbrace{2020...20}^{a\times 20}21$ & $\overbrace{00...0}^{2a\times 0}00$ & $\overbrace{2020...20}^{(a-1)\times 20}2110$ \\
 \hline
 $\overbrace{2020...20}^{a\times 20}221$ & $\overbrace{00...0}^{2a\times 0}100$ & $\overbrace{2020...20}^{(a-1)\times 20}2101210$ \\
 \hline
 $\overbrace{2020...20}^{a\times 20}2221$ & $\overbrace{00...0}^{2a\times 0}1100$ & $\overbrace{2020...20}^{(a-2)\times 20}21022210$ \\
 \hline
 $\overbrace{2020...20}^{a\times 20}211$ & $\overbrace{00...0}^{2a\times 0}010$ & $\overbrace{2020...20}^{(a-1)\times 20}22020$ \\
 \hline
 $\overbrace{2020...20}^{a\times 20}2211$ & $\overbrace{00...0}^{2a\times 0}1010$ & $\overbrace{2020...20}^{(a-2)\times 20}21020020$ \\
 \hline
 $\overbrace{2020...20}^{a\times 20}22211$ & $\overbrace{00...0}^{2a\times 0}11010$ & $\overbrace{2020...20}^{(a-2)\times 20}211000020$ 
\end{tabular}

\bigskip

In any case, when $a\geq 2$ adding an extra $20$ block increases the digit sums of $x$ and $23x+(3^r-1)/2$ by $2$. So, by induction, we may assume that $a\leq 2$. It remains to check the cases where $a\leq 2$, which can be done by a computer search.
\end{proof}

\begin{cor}
Let $r\geq 2$ and $0\leq x <3^r$ an integer such that $x\not\equiv 2$ or $6$ ({\rm {mod}} $9$). Then
 $$
\left[23x+\frac{3^r-1}{2}\right] \leq \lbr x\rbr +\left[2x+2+\frac{3^r-1}{2}\right] +2.
$$
\end{cor}

\begin{proof}
 If $x\not\equiv 2$ or $6$ (mod $9$), then $2x+\frac{3^r-1}{2}\not\equiv 7$ or $8$ (mod $9$), and therefore 
$$\left[2x+\frac{3^r-1}{2}\right] \leq\left[2x+2+\frac{3^r-1}{2}\right] .$$
\end{proof}

We can now finish the proof of theorem \ref{mainfinite}. Let $r\geq 2$ and let $0< x <3^r-1$ be an integer such that $x\not\equiv 2$ or $6$ (mod $9$). Then $2x+2+(3^r-1)/2<3^{r+1}$, and 
$$\left[3^{r+1}-1-\bigl(2x+2+\frac{3^r-1}{2}\bigr)\right] =2(r+1)-\left[2x+2+\frac{3^r-1}{2}\right] .$$ 
By the previous corollary, we have
$$
\left[23x+\frac{3^r-1}{2}\right] +\left[3^{r+1}-3-2x-\frac{3^r-1}{2}\right] \leq \lbr x\rbr +2r+4.
$$
Hence,
$$
\left[23x+\frac{3^r-1}{2}\right] _r+\left[-2x-\frac{3^r-1}{2}\right] _r=
$$
$$
=\left[23x+\frac{3^r-1}{2}\right] _r+\left[3(3^{r}-1)-2x-\frac{3^r-1}{2}\right] _r\leq
$$
$$
\leq\left[23x+\frac{3^r-1}{2}\right] +\left[3^{r+1}-3-2x-\frac{3^r-1}{2}\right] \leq$$
$$\leq \lbr x\rbr +2r+4=[x\rbr _r+2r+4.
$$
Using \cite[Lemma 2.10]{Ka-RL-T}, we conclude (as in the proof of \cite[Theorem 2.12]{Ka-RL-T}) that 
$$
\left[23x+\frac{3^r-1}{2}\right]_r+\left[-2x-\frac{3^r-1}{2}\right]_r\leq \lbr x\rbr _r+2r
$$
or, equivalently,
$$
\left[23x+\frac{3^r-1}{2}\right] _r\leq \lbr x\rbr _r+\left[2x+\frac{3^r-1}{2}\right] _r
$$

If $x\equiv 2$ or $6$ (mod $9$), then either $x=2020\ldots20_3=3(3^r-1)/4$ or $x=0202\ldots02_3=(3^r-1)/4$, in which case the inequality is trivial, or we can multiply $x$ by a suitable power of $3$ (which cyclically permutes its digits modulo $3^r-1$) to obtain an $x$ which is not $\equiv 2$ or $6$ (mod $9$), to which we can apply the previous argument.

\section{Determination of the monodromy groups}
In this section, we will determine the monodromy of $\sH(\psi,23,4)(1)$ in characteristic $p=3$. In this case, the field $\F_3(\mu_{23})$ is the field $F_{3^{11}}$. So by Lemma \ref{det}, the arithmetic monodromy group lies in $\SL_{23}(\C)$. [In fact it lies in $\SO_{23}(\C)$, since it is an irreducible subgroup with real (in fact integer) traces.]

Looking at the Frobenius $Frob_{-1,\F_3}$ at the point $t=-1$, we have, by computer calculation, that its first seven powers have traces 
$$0,-2,0,2,0,-2, 7.$$

First we prove the following theorem on finite subgroups of $\SL_{23}(\CC)$:

\begin{thm}\label{simple}
Let $V = \CC^{23}$ and let $G < \SL(V)$ be a finite irreducible subgroup. Let $\chi$ denote the character of $G$ 
afforded by $V$, and suppose that all the following conditions hold:
\begin{enumerate}[\rm(i)]
\item $\chi$ is real-valued;
\item $\chi$ is primitive;
\item $G$ contains an element $\gamma$ such that the traces of $\gamma, \gamma^2, \ldots, \gamma^7$ acting on $V$
are $0,-2,0,2,0,-2, 7.$
\end{enumerate}
Then $G \cong \Co_2$ in its unique (orthogonal) irreducible representation of degree $23$.
\end{thm}

\begin{proof}
By the assumption, the $G$-module $V$ is irreducible and primitive; furthermore, it is tensor indecomposable and not tensor induced
since $\dim V = 23$ is prime. Next, we observe by Schur's Lemma that condition (i) implies $\bfZ(G) = 1$. Now we can apply 
\cite[Proposition 2.8]{G-T} (noting that the subgroup $H$ in its proof is just $G$ since $G < \SL(V)$) and arrive at one of the following
two cases.

\smallskip
(a) {\it Extraspecial case}: $P \lhd G$ for some extraspecial $23$-group of order $23^3$ that acts irreducibly on $V$. But in this case,
$\chi|_P$ cannot be real-valued (in fact, $\Q(\chi|_P)$ would be $\Q(\exp(2\pi i/23))$, violating (i).

\smallskip
(b) {\it Almost simple case}: $S \lhd G \leq \Aut(S)$ for some finite non-abelian simple group $S$. In this case, we can apply the main
result of \cite{H-M} and arrive at one of the following possibilities for $S$.

$\bullet$ $S = \Alt_{24}$, $M_{24}$, or $\PSL_2(23)$. Correspondingly, we have that
$G = \Alt_{24}$ or $\SSS_{24}$, $M_{24}$, and $\PSL_2(23)$ or $\PGL_2(23)$. In all of these possibilities, 
$\chi(x) \geq -1$ for $x \in G$ by \cite{ATLAS}, violating (iii).

$\bullet$ $S = \PSL_2(47)$. This is ruled out since $\Q(\chi|_S)$ would be $\Q(\sqrt{-47})$, violating (i).    

$\bullet$ $S = \Co_3$. In this case, $G = \Co_3$, and this possibility is ruled out by the existence of the (conjugacy class of the) element $\gamma$ in (iii).  Now $\gamma^7$ 
has trace $7$. In  $\Co_3$, the only class with trace $7$ is class $2$ in
{\tt Magma} notation. The only class $\gamma$ with $\gamma^7$  in class $2$ is either class $2$ or class 
$29$. But $\gamma$ cannot be in class $2$, because its
trace is 0, not 7. So $\gamma$ must be in class $29$, which does have trace 0. However, the square of class $29$ in $\Co_3$
 is class $16$, whose trace is $2$, not $ -2$. Therefore we do not have $\Co_3$.

$\bullet$ $S = \Co_2$. In this case $G = \Co_2$, as stated. 
\end{proof}

\begin{thm}In characteristic $3$, the rigid local system $\sH(\psi,23, 4)(1)$  on $\G_m/\F_3$ has $G_{geom} =G_{arith}=\Co_2$. 
\end{thm}

\begin{proof}We know that $[23]^\star \sH(\psi,23, 4)(1)$ is not geometrically induced, so a fortiori $\sH(\psi,23, 4)(1)$ is not geometrically induced, and hence  $\sH(\psi,23, 4)(1)$ is not arithmetically induced (primitivity passes to overgroups). We know that $\sH(\psi,23, 4)(1)$ is geometrically irreducible, hence all the more arithmetically irreducible. By the Determinant Lemma \ref{det}, and the finiteness theorem, we know that $\sH(\psi,23, 4)(1)$ has integer traces and trivial determinant. In view of the Theorem \ref{simple}, this forces $G_{arith}$ to
be $\Co_2$. Now $G_{geom}$ is a normal subgroup of $G_{arith}$, and it is nontrivial because it is itself irreducible. But $\Co_2$ is a simple group, so $G_{geom}$ must be $\Co_2$.
\end{proof}

\begin{cor}For any integer $M \ge 1$ prime to $3$, the Kummer pullback $[M]^\star \sH(\psi,23, 4)(1)$ on $\G_m/\F_3$ has $G_{geom} =G_{arith}=\Co_2$. 
\end{cor}
\begin{proof}By Lemma \ref{pullback}, $[M]^\star \sH(\psi,23, 4)(1)$ is geometrically irreducible. Its $G_{geom}$ is a then a normal subgroup of the $G_{geom}$ of $\sH(\psi,23, 4)(1)$, namely  $\Co_2$, with a quotient cyclic of order dividing $M$. Because $\Co_2$ is simple and nonabelian, this quotient must be trivial. Thus $[M]^\star \sH(\psi,23, 4)(1)$ on $\G_m/\F_3$ has $G_{geom} =\Co_2$. From Theorem \ref{simple}, we see that $\Co_2$ is maximal (and from \cite{ATLAS} that it is minimal as well, although we will not use this) among finite irreducible subgroups of $SO(23,\C)$. Thus the $G_{arith}$ of the pullback must itself be $\Co_2$.
\end{proof}

\end{document}